\documentclass[11pt]{amsart}

\usepackage{eucal,fullpage,times,amsmath,amsthm,amssymb,mathrsfs,stmaryrd,color,enumerate,accents,enumitem}
\usepackage[all]{xy}
\usepackage{url}

\usepackage{extarrows}
\usepackage{xypic}
\usepackage{exscale}
\usepackage{relsize}

\usepackage[colorlinks=true,hyperindex, citecolor=cyan, urlcolor=cyan, pagebackref=true]{hyperref}

\newtheorem{theorem}{Theorem}[section]
\newtheorem{lemma}[theorem]{Lemma}

\theoremstyle{definition}

\newtheorem{example}[theorem]{Example}

\newtheorem{proposition}[theorem]{Proposition}
\newtheorem{corollary}[theorem]{Corollary}
\newtheorem{remark}[theorem]{Remark}
\newtheorem{conjecture}[theorem]{Conjecture}

\theoremstyle{remark}

\newcommand{\be}{\begin{equation}}
\newcommand{\ee}{\end{equation}}

\numberwithin{equation}{section}

\begin{document}
\title{Characteristic numbers, Jiang subgroup and non-positive curvature}

\author{Ping Li}
\address{School of Mathematical Sciences, Fudan University, Shanghai 200433, China}

\address{School of Mathematical Sciences, Tongji University, Shanghai 200092, China}

\email{pinglimath@gmail.com\\
pingli@tongji.edu.cn}
%\author{}
%\address{}
%\email{}
\thanks{The author was partially supported by the National
Natural Science Foundation of China (Grant No. 11722109).}

%    General info
 \subjclass[2010]{57R20, 58C30, 32Q05.}

%\date{January 1, 2001 and, in revised form, June 22, 2001.}

\dedicatory{Dedicated to Professor Boju Jiang on the occasion of his 85th birthday.}
\keywords{Euler characteristic, signature, Chern number, $\chi_y$-genus, Jiang subgroup, aspherical manifold, non-positive curvature}

\begin{abstract}
By refining an idea of Farrell, we present a sufficient condition in terms of the Jiang subgroup for the vanishing of signature and Hirzebruch's $\chi_y$-genus on compact smooth and K\"{a}hler manifolds respectively. Along this line we show that the $\chi_y$-genus of a non-positively curved compact K\"{a}hler manifold vanishes when the center of its fundamental group is non-trivial, which partially answers a question of Farrell. Moreover, in the latter case all the Chern numbers vanish whenever its complex dimension is no more than $4$, which also provides some evidence to a conjecture proposed by the author and Zheng.
\end{abstract}

\maketitle

\section{Introduction and main results}\label{introduction}
Unless otherwise stated, all smooth and complex manifolds considered in this article are compact and connected, and the dimension of a complex manifold is referred to its \emph{complex} dimension.

Let $X$ be a finite CW-complex and $f$ a self-map of $X$. A homotopy $h_t:X\rightarrow X$ ($0\leq t\leq1$) is called a \emph{cyclic homotopy} of $f$ if $h_0=h_1=f$. In this case the trace $h_t(x_0)$ at any $x_0\in X$ is a loop at $f(x_0)$ and hence represents an element $[h_t(x_0)]\in\pi_1\big(X,f(x_0)\big)$. The trace subgroup of cyclic homotopies $J(f,x_0)\leq \pi_1\big(X,f(x_0)\big)$ is defined by
\be\begin{split}\label{jiangsubgroup}
J(f,x_0)
:=\big\{\xi\in\pi_1\big(X,f(x_0)\big)~\big|~
\text{there exists a cyclic homotopy $h_t$ of $f$ such that $[h_t(x_0)]=\xi$}\big\}.
\end{split}\ee

This subgroup $J(f,x_0)$ was introduced by Boju Jiang in \cite{Ji1} to compute the Nielsen number of self-maps, which is a refined version of the classical Lefschetz number, and is now called \emph{Jiang subgroup} and becomes an indispensable tool in Nielsen fixed point theory. We refer the reader to Jiang's book \cite[\S2.3-\S2.6]{Ji2} for various properties of Jiang subgroup and its applications to Nielsen fixed point theory. The Jiang subgroup $J(f,x_0)$ is indeed a homotopy type invariant and independent of the choice of the basepoint $x_0$ up to an isomorphism which is compatible with $\pi_1(X,x_0)$ (\cite[Lemma 3.9]{Ji2}). Therefore we can simply denote it by $J(f)$.

In this article we are concerned with the Jiang subgroup $J(X):=J(\text{id}_X)$ with respect to the identity map. It turns out that $J(X)\leq\text{Z}\big(\pi_1(X)\big)$, the center of $\pi_1(X)$ (\cite[Lemma 3.7]{Ji2}), and the Euler characteristic $\chi(X)=0$ provided that $J(X)$ is non-trivial (\cite[Prop. 4.12]{Ji2}). The Jiang subgroup $J(X)$ was also investigated in detail by Gottlieb in \cite{Go} and many of Jiang's results concerning $J(X)$ were reproved there. Nevertheless, an important new result obtained by Gottlieb is that $J(X)=\text{Z}\big(\pi_1(X)\big)$ whenever $X$ is \emph{aspherical} (\cite[p. 848]{Go}), i.e., the universal cover of $X$ is contractible.

Combining the above-mentioned facts implies the following useful fact.
\begin{theorem}[Jiang, Gottlieb]\label{JG}
The Euler characteristic of a finite aspherical CW-complex $X$ vanishes whenever $\text{Z}\big(\pi_1(X)\big)$ is nontrivial.
\end{theorem}

Farrell showed in \cite{Fa} that more characteristic numbers of aspherical manifolds vanish if an additional restriction is imposed on the fundamental group. To be more precise, he showed the following theorem (\cite[Thms 1.2, 2.1]{Fa}).

\begin{theorem}[Farrell]\label{Farrelltheorem}
Let $M$ be an aspherical manifold. If $\text{Z}\big(\pi_1(M)\big)$ is nontrivial and $\pi_1(M)$ is residually finite, then the signature $\text{sign}(M)=0$. If moreover $M$ is K\"{a}hler, then the $\chi_y$-genus $\chi_y(M)=0$.
\end{theorem}
More details about signature and the $\chi_y$-genus can be found in Section \ref{preliminaries}. Recall that a group $G$ is called \emph{residually finite} if for every $g\in G\backslash\{1\}$, there exists a homomorphism from $G$ to a finite group such that the image of $g$ is not the identity (\cite[\S 2]{CC}). It is now well-known that the fundamental groups of closed surfaces, 3-manifolds and hyperbolic manifolds are residually finite.

By refining the proof of Farrell in \cite{Fa}, our first result is the following theorem.
\begin{theorem}\label{maintheorem}
Let $M$ be a smooth manifold. If the Jiang subgroup $J(M)$ satisfies
\be\label{Jiang condition}J(M)\not\subset\bigcap_{\overset{H\lhd \pi_1(M)}{[\pi_1(M):H]<\infty}}H,\ee
i.e., $J(M)$ is not contained in the intersection of all normal subgroups of finite index of $\pi_1(M)$,
then the signature $\text{sign}(M)=0$. If moreover $M$ is K\"{a}hler, then the $\chi_y$-genus $\chi_y(M)=0$.
\end{theorem}
\begin{remark}\label{remark}
\begin{enumerate}
\item
It turns out that $\chi_y(M)\big|_{y=-1}=\chi(M)$ and $\chi_y(M)\big|_{y=1}=\text{sign}(M)$ (see Section \ref{preliminaries} for more details). Therefore the vanishing of $\chi_y(M)$ contains more information than those of the Euler characteristic and the signature.

\item
One may wonder, under the same hypothesis as in Theorem \ref{maintheorem} on \emph{spin} manifolds, if we can deduce the vanishing of the $\hat{A}$-genus. However, the strategy of the proof in treating the signature and the $\chi_y$-genus can \emph{not} be carried over to the case of the $\hat{A}$-genus on spin manifolds, which will be clear in Section \ref{proof of maintheorem} (see Remark \ref{remark2}).
\end{enumerate}
\end{remark}

In order to present some useful sufficient conditions to the hypothesis condition (\ref{Jiang condition}), let us recall two related facts in group theory. The residual finiteness of a group $G$ has several equivalent definitions, one of which is that the intersection of all normal subgroups of finite index is trivial (\cite[Prop. 2.1.11]{CC}):
\be\label{residualfiniteness}\bigcap_{\overset{H\lhd G}{[G:H]<\infty}}H=\{1\}.\ee
Another fact is the so-called Poincar\'{e} Theorem in group theory, which states that any subgroup of finite index in $G$ contains a normal subgroup of $G$ which is also of finite index (\cite[Lemma 2.1.10]{CC}). This latter fact, together with (\ref{residualfiniteness}), also implies that the residual finiteness has another equivalent characterization: the intersection of all subgroups of finite index is trivial.

With these facts in mind,
The hypothesis condition (\ref{Jiang condition}) becomes simpler under two extreme cases: the right-hand side (RHS for short) of (\ref{Jiang condition}) is trivial or its LHS is the whole $\pi_1(M)$. When the RHS of (\ref{Jiang condition}) is trivial, i.e., $\pi_1(M)$ is residually finite, the hypothesis (\ref{Jiang condition}) means that $J(M)$ is non-trivial. When $J(M)=\pi_1(M)$, the hypothesis (\ref{Jiang condition}) means that $\pi_1(M)$ contains a proper subgroup of finite index. The hypothesis $J(M)=\pi_1(M)$ is particularly useful in computing the Nielsen number (\cite[Thm 4.2]{Ji2}) and two family of spaces satisfy it: (the homotopy types of ) $H$-spaces and homogeneous spaces (\cite[Thm 3.11]{Ji2}).

We summarize the discussions above into the following consequence.
\begin{corollary}
The conclusions in Theorem \ref{maintheorem} hold true as long as one of the following two situations occurs.
\begin{enumerate}
\item
The Jiang subgroup $J(M)$ is non-trivial and $\pi_1(M)$ is residually finite. In particular for aspherical manifolds this is Farrell's Theorem \ref{Farrelltheorem} as in this case $J(M)=\text{Z}\big(\pi_1(M)\big)$ by Theorem \ref{JG}.

\item
$J(M)=\pi_1(M)$ and contains a proper subgroup of finite index. In particular, the former holds when $M$ has the homotopy type of an $H$-space or a homogeneous space, and the latter holds when $\pi_1(M)$ is a non-trivial finite group.
\end{enumerate}
\end{corollary}
Recall that a non-positive sectional curvature (\emph{non-positively curved} for short) Riemannian manifold is aspherical, thanks to the Cartan-Hadamard Theorem (\cite[p. 59]{Zhe}). Such a condition is usually regarded as the differential-geometry counterpart to that of asphericity. Farrell conjectured that the condition of $\pi_1(M)$ being residually finite in Theorem \ref{Farrelltheorem} may be superfluous (\cite[p. 165]{Fa}). Our second result is to partially answer Farrell's question in the K\"{a}hler setting, but with non-positive sectional curvature in place of asphericity.

The notions of K\"{a}hler hyperbolicity and K\"{a}hler non-ellipticity were introduced by Gromov et al (\cite{Gr}, \cite{CX}, \cite{JZ}) to attack the well-known Singer Conjecture (\cite[\S11]{Lu}), which include negatively and non-positively curved K\"{a}hler manifolds respectively. More details can be found in Section \ref{proof of vanish1}. Our second result is the following theorem.
\begin{theorem}\label{vanish1}
Let $M$ be a K\"{a}hler non-elliptic manifold. Then the Euler characteristic of $M$ vanishes if and only if $\chi_y(M)=0$. In particular, a non-positively curved K\"{a}hler manifold $M$ with non-trivial $\text{Z}\big(\pi_1(M)\big)$ has vanishing $\chi_y$-genus: $\chi_y(M)=0$.
\end{theorem}
Indeed we shall show in Proposition \ref{vanish1-prop} that the vanishing of the Euler characteristic is equivalent to that of the $\chi_y$-genus on those manifolds satisfying the Singer Conjecture, from which Theorem \ref{vanish1} follows. The conclusions in Theorem \ref{vanish1} can be further strengthened whenever the dimension is small. This is the following result.
\begin{theorem}\label{vanish2}
Let $M$ be an $n$-dimensional non-positively curved K\"{a}hler manifold with non-trivial $\text{Z}\big(\pi_1(M)\big)$. Then all the Chern numbers of $M$ vanish when the dimension $n\leq4$.
\end{theorem}
%The main strategy in proving Theorem \ref{vanish2} is to show a vanishing-type result (Theorem \ref{vanish2-thm}) by extending an idea of Qi Zhang (\cite{Zha}).
In \cite{LZ} the author and F.Y. Zheng obtained some Chern class/number inequalities on numerically effective holomorphic vector bundles and gave some related applications to non-negative and non-positive curvature manifolds in various senses (\cite[\S 3]{LZ}). Motivated by these results, we proposed the following conjecture (\cite[\S4]{LZ}), which can be regarded as the complex analogue to the classical Hopf conjecture.
\begin{conjecture}[Li-Zheng]\label{question}
Let $M$ be an $n$-dimensional K\"{a}hler manifold
with non-positive holomorphic bisectional curvature whose Ricci curvature is quasi-negative. Then the signed Euler characteristic  $(-1)^n\chi(M)>0$.
\end{conjecture}
Conjecture \ref{question} is true when $n=2$ (\cite[Prop. 4.4]{LZ}). It is well-known that the sign of holomorphic bisectional curvature is dominated by that of Riemannian sectional curvature (\cite[p. 178]{Zhe}). We shall explain in Section \ref{proof of vanish2} that our proof of Theorem \ref{vanish2} indeed leads to the following corollary, which provides some more evidences to Conjecture \ref{question}.
\begin{corollary}\label{vanish3}
Conjecture \ref{question} is true for $n$-dimensional non-positively curved K\"{a}hler manifolds whose Ricci curvature is quasi-negative when $n\leq4$.
\end{corollary}

The rest of this article is organized as follows. We briefly collect in Section \ref{preliminaries} some basic facts on the signature and $\chi_y$-genus in the form we shall use them in this article. Sections \ref{proof of maintheorem}, \ref{proof of vanish1} and \ref{proof of vanish2} are devoted to the proofs of Theorems \ref{maintheorem}, \ref{vanish1} and \ref{vanish2} respectively. In the last Section \ref{appendix} entitled ``Appendix", some non-standard facts mentioned in Section \ref{preliminaries} will be proved or further explained for the completeness as well as for the reader's convenience.

\section{Preliminaries}\label{preliminaries}
\subsection{Signature and the $\chi_y$-genus}
We recall in this subsection some basic facts on the signature and Hirzebruch's $\chi_y$-genus.

Let $M$ be an oriented smooth manifold. Denote by $\text{sign}(M)$ the signature of $M$, which is defined to be the index of the intersection pairing on the middle-dimensional real cohomology when the dimension of $M$ is divisible by $4$ and zero otherwise. $\text{sign}(M)$ is a rationally linear combination of Pontrjagin numbers and hence a characteristic number, thanks to Hirzebruch's Signature Theorem (\cite{Hi}).
It is well-known that (\cite[\S6]{AS}) $\text{sign}(M)$ can be realized as the index of an elliptic operator $D$ on $M$: $$\text{sign}(M)=\dim_{\mathbb{C}}\ker(D)-\dim_{\mathbb{C}}\text{coker}(D),$$
which is called the \emph{signature operator} and $\ker(D)$ and $\text{coker}(D)$ are both subspaces of the (complexified) de-Rham cohomology groups on $M$. The celebrated Atiyah-Singer index theorem then provides an alternative proof to the above-mentioned Hirzebruch Signature Theorem.

We only briefly recall Hirzebruch's $\chi_y$-genus for compact complex manifolds, which is enough for our purpose in this article. We refer the reader to \cite{Li0} and \cite[\S3]{Li1} for its definition on general almost-complex manifolds and the summary on its various properties.

Let $M=(M,J)$ be a complex manifold with (complex) dimension $n$ and $\bar{\partial}$ the $d$-bar operator which
acts on the complex vector spaces $\Omega^{p,q}(M)$ ($0\leq p,q\leq
n$) of $(p,q)$-type differential forms on $(M,J)$ in the sense
of $J$. For each integer $0\leq p\leq n$, we have the following
Dolbeault-type elliptic complex
\be\label{DC}0\rightarrow\Omega^{p,0}(M)
\xrightarrow{\bar{\partial}}\Omega^{p,1}(M)
\xrightarrow{\bar{\partial}}\cdots\xrightarrow{\bar{\partial}}\Omega^{p,n}(M)\rightarrow
0,\ee whose index is denoted by $\chi^{p}(M)$ in the notation of Hirzebruch (\cite{Hi}):
\be\label{Hodgenumber}\chi^{p}(M):=\sum_{q=0}^{n}(-1)
^{q}\text{dim}_{\mathbb{C}}H^{p,q}
_{\bar{\partial}}(M)=\sum_{q=0}^{n}(-1)^{q}h^{p,q}(M),\ee
where
$H^{p,q}
_{\bar{\partial}}(M)$ are the Dolbeault cohomology groups and their dimensions $h^{p,q}(M)$ are the usual Hodge numbers of $M$. Note that $\chi^0(M)$ is the \emph{Todd genus} of $M$, which is  sometimes called the arithmetic genus in algebraic geometry.

The Hirzebruch's
$\chi_{y}$-genus, introduced by Hirzebruch in \cite{Hi} and denoted by $\chi_{y}(M)$, is the generating
function of these indices $\chi^p(M)$:
$$\chi_{y}(M):=\sum_{p=0}^{n}\chi^{p}(M)\cdot y^{p},$$
The general form of the Hirzebruch-Riemann-Roch theorem, which was established by Hirzebruch (\cite{Hi}) for projective manifolds and by Atiyah and Singer in the general setting (\cite[\S4]{AS}), allows us to express these $\chi^p(M)$ and thus
$\chi_y(M)$ in terms of Chern numbers of $M$ as follows
\be\label{HRR}\chi_y(M)=\mathlarger{\int}_M
\prod_{i=1}^n\frac{x_i(1+ye^{-x_i})}{1-e^{-x_i}},\ee
where $x_1,\ldots,x_n$ are Chern roots of $(M,J)$, i.e., the
$i$-th elementary symmetric polynomial of $x_1$, $\ldots$, $x_n$ represents the
$i$-th Chern class of $(M,J)$. Thus these $\chi^p(M)$ are also characteristic numbers. The formula (\ref{HRR}) implies that $\chi_y(M)\big|_{y=-1}=\chi(M)$, and, together with Hirzebruch's Signature Theorem, implies that $\chi_y(M)\big|_{y=1}=\text{sign}(M)$. Therefore the vanishing of $\chi_y(M)$ implies those of the Euler characteristic and the signature, as stated in Remark \ref{remark}.

\subsection{The $\chi_y$-genus and $L^2$-Hodge numbers}\label{read}
we briefly recall in this subsection that how to express the indices $\chi^p$ in (\ref{Hodgenumber}) in terms of the $L^2$-Hodge numbers for our later purpose. For the reader's convenience more related details are included in Section \ref{moredetailsonL2Hodgenumbers}. The reader may also consult \cite[\S4]{Li1}. A thorough treatment on these materials can be found in \cite[\S1]{Lu} or \cite[\S1]{Gr}.

Assume that $M=(M,g,J)$ is an $n$-dimensional complex manifold equipped with a Hermitian metric $g$, and $$\pi:~(\widetilde{M},\widetilde{g},\widetilde{J})\longrightarrow
(M,g,J)$$
its \emph{universal} cover with $\pi_1(M)$ played as an isometric group of deck transformations.

Let $\mathcal{H}^{p,q}_{(2)}(\widetilde{M})$ be the spaces of $L^2$-harmonic $(p,q)$-forms on $L^2\Omega^{p,q}(\widetilde{M})$, the squared integrable $(p,q)$-forms on $(\widetilde{M},\widetilde{g})$, and denote by $$\text{dim}_{\pi_1(M)}\mathcal{H}^{p,q}_{(2)}(\widetilde{M})$$
the \emph{Von Neumann dimension} of $\mathcal{H}^{p,q}_{(2)}(\widetilde{M})$ with respect to $\pi_1(M)$, which is a \emph{nonnegative real number} in our situation. Its precise definition is not important in our article and we refer the reader to \cite{Lu} for more details.
The $L^2$-Hodge numbers of $M$, denoted by $h^{p,q}_{(2)}(M)$, are defined to be
$$h^{p,q}_{(2)}(M):=\text{dim}_{\pi_1(M)}
\mathcal{H}^{p,q}_{(2)}(\widetilde{M})
\in\mathbb{R}_{\geq0},\qquad(0\leq p,q\leq n).$$
Like the situation of the usual Hodge numbers, it turns out that $h^{p,q}_{(2)}(M)$ are independent of the Hermitian metric $g$ and depend only on $(M,J)$.

The following fact is an application of Atiyah's $L^2$-index theorem (\cite{At}).
\begin{lemma}\label{L2expression}
These $\chi^p(M)$ can be similarly expressed in terms of the $L^2$-Hodge numbers $h^{p,q}_{(2)}(M)$ as follows
\be\label{factchiphodgenumbers}
\chi^p(M)=\sum_{q=0}^n(-1)^qh^{p,q}_{(2)}(M)\qquad 0\leq p\leq n.\ee
\end{lemma}
Lemma \ref{L2expression} is well-known to experts. Nevertheless, for the reader's convenience as well as for the completeness,
We shall indicate its proof in the Appendix, Section \ref{moredetailsonL2Hodgenumbers}.

\section{Proof of Theorem \ref{maintheorem}}\label{proof of maintheorem}
The hypothesis condition (\ref{Jiang condition}) in Theorem \ref{maintheorem} leads to the following consequence.
\begin{lemma}\label{technicalproposition}
Let $X$ be a connected finite CW-complex and its Jiang subgroup $J(X)$ satisfies
\be\label{Jiang condition2}J(X)\not\subset\bigcap_{\overset{H\lhd \pi_1(X)}{[\pi_1(X):H]<\infty}}H.\ee
Then there exist a finite-sheeted cover $\widetilde{X}$ of $X$ and its deck transformation $\tau$ of finite order such that $\tau$ is fixed-point-free and homotopic to the identity map. If moreover $X$ is a smooth (resp. complex) manifold, then this deck transformation $\tau$ is smooth (resp. holomorphic) either.
\end{lemma}

We first show that how to apply Lemma \ref{technicalproposition} to prove Theorem \ref{maintheorem}, which essentially follows the strategy of Farrell in \cite{Fa}, and postpone its proof to the end of this section.

If $M$ is a compact, connected and oriented smooth (resp. K\"{a}hler) manifold with a finite-sheeted cover $\widetilde{M}\stackrel{f}{\longrightarrow}M$, then $\widetilde{M}$ is also a \emph{compact} connected smooth (resp. K\"{a}hler) manifold equipped with the induced orientation. Since characteristic numbers are multiplicative with respect to finite covers, we have
$$\text{sign}(\widetilde{M})=\text{deg}(f)\cdot\text{sign}(M)$$
and
$$\chi_y(\widetilde{M})=\text{deg}(f)\cdot\chi_y(M)$$
respectively. Together with Lemma \ref{technicalproposition}, Theorem \ref{maintheorem} follows from the following fact.
\begin{lemma}
Let $M$ be an oriented smooth (resp. K\"{a}hler) manifold equipped with a fixed-point-free smooth (resp. holomorphic) finite cyclic group action. If some generator of this group action is homotopic to the identity map, then $\text{sign}(M)=0$ (resp. $\chi_y(M)=0$.)
\end{lemma}
\begin{proof}
Let the cyclic group be $G$ and a generator $g\in G$ is homotopic to the identity map.

Since both $\ker(D)$ and $\text{coker}(D)$ of the signature operator $D$ are subspaces of the de-Rham cohomology groups of $M$, homotopy invariance of de-Rham cohomology implies that
the equivariant indices of $D$ at $g$ and the identity element are indeed equal:
\be
\begin{split}
\text{sign}(g,M)&=
\text{Trace}(g\big|_{\ker(D)})-
\text{Trace}(g\big|_{\text{coker}(D)})\\
&=\text{Trace}(\text{id}\big|_{\ker(D)})-
\text{Trace}(\text{id}\big|_{\text{coker}(D)})\qquad(\text{$g$ is homotopic to the identity map})\\
&=\dim_{\mathbb{C}}\ker(D)-\dim_{\mathbb{C}}\text{coker}(D)\\
&=\text{sign}(M).
\end{split}
\nonumber
\ee

On the other hand, the Atiyah-Singer $G$-signature Theorem (\cite[Thm 6.12]{AS}) implies that equivariant index $\text{sign}(g,M)=0$ as the fixed-point set of this $G$-action is empty, which leads to the desired conclusion that $\text{sign}(M)=0$.

When $M$ is K\"{a}hler, we can define the Lefschetz number $\chi^p(g,M)$ of the elliptic complex (\ref{DC}) at $g$. Hodge theory tells us that for \emph{K\"{a}hler} manifolds the Dolbeault cohomology groups $H^{p,q}_{\bar{\partial}}(M)$ can be canonically viewed as subspaces of de-Rham cohomology groups and hence have the homotopy invariance (see \cite[Lemma 2.3]{Fa} for a slightly different treatment at this point). This implies that
\be
\begin{split}
\chi^p(g,M)&=
\sum_{q=0}^n(-1)^q\cdot\text{Trace}(g\big|_{H^{p,q}_{\bar{\partial}}(M)})
\\
&=\sum_{q=0}^n(-1)^q\cdot\text{Trace}(\text{id}\big|_{H^{p,q}_{\bar{\partial}}(M)})\qquad(\text{$g$ is homotopic to the identity map})\\
&=\sum_{q=0}^n(-1)^q\cdot\dim_{\mathbb{C}}H^{p,q}_{\bar{\partial}}(M)\\
&=\chi^p(M).
\end{split}\nonumber
\ee

On the other hand, the Lefschetz fixed-point formula of Atiyah-Bott-Singer, which was first treated by Atiyah and Bott in the isolated case in \cite{AB} and then by Atiyah and Singer in \cite[\S4]{AS} in the general setting, yields that $\chi^p(g,M)=0$ as this $G$-action has no fixed points. This leads to the desired conclusion that these $\chi^p(M)=0$ and hence $\chi_y(M)=0$.
\end{proof}

\begin{remark}\label{remark2}
\begin{enumerate}
\item
We can see from the proof of this lemma that if a compact \emph{connected} Lie group $G$ acts smoothly (resp. holomorphically) on an oriented smooth (resp. K\"{a}hler) manifold, then the equivariant index $\text{sign}(g,M)$ \big(resp. $\chi^p(g,M)$\big) is a constant:
$$\text{sign}(g,M)\equiv\text{sign}(M),
\qquad\big(\text{resp. $\chi^p(g,M)\equiv\chi^p(M)$},\big)\qquad\forall g\in G.$$
This observation is indeed the starting point of the rigidity phenomena of the Dolbeault complexes (\ref{DC}) on general complex manifolds (\cite{Ko}) and of the Dirac operator on spin manifolds (\cite{AH}). The latter in turn motivates the rigidity of the elliptic genera (\cite{Ta}, \cite{BT}, \cite{Liu1}, \cite{Liu2}).

\item
On the other hand, if a spin manifold admits a diffeomorphism which is homotopic to the identity map, its equivariant index may \emph{not} equal to the index of the Dirac operator as in this situation the kernel and cokernel of the Dirac operator in general have \emph{no} properties of homotopy invariance. So even if assuming the same hypotheses as in Theorem \ref{maintheorem} on spin manifolds, we cannot deduce the vanishing of the $\hat{A}$-genus.
\end{enumerate}
\end{remark}

\subsection{Proof of Lemma \ref{technicalproposition}}
The proof is based on some standard facts in covering space theory and a standard reference is \cite[\S1.3]{Ha}.
\begin{proof}
Let us choose once for all a basepoint $x_0$ in $X$.  By the hypothesis (\ref{Jiang condition2}), there exist an $\alpha\in J(X,x_0)$ and a normal subgroup $H$ of finite index in $\pi_1(X,x_0)$ such that $\alpha\not\in H$. Let $(\widetilde{X}, \widetilde{x}_0)\stackrel{p}{\longrightarrow}(X,x_0)$ be the finite-sheeted covering map which corresponds to $H$ and thus $p_{\ast}\big(\pi_1(\widetilde{X},\widetilde{x}_0)\big)=H$ (\cite[Prop. 1.36]{Ha}). By the definition of (\ref{jiangsubgroup}), there exists a homotopy $h_t: X\rightarrow X$ such that $h_0=h_1=\text{id}_X$ and $[h_t(x_0)]=\alpha$. By the homotopy lifting property (\cite[Prop. 1.30]{Ha}), there exists a unique homotopy $\widetilde{h}_t:\widetilde{X}\rightarrow\widetilde{X}$ with $p\circ\widetilde{h}_t=h_t\circ p$ and $\widetilde{h}_0=\text{id}_{\widetilde{X}}$, i.e., we have the following commutative diagram:
\be
 \xymatrix{\widetilde{X} \ar[rr]^-{\widetilde{h}_t} \ar[d]_-{p} && \widetilde{X} \ar[d]^-{p} \\
		  X \ar[rr]^-{h_t}& & X },\qquad\widetilde{h}_0=\text{id}_{\widetilde{X}}. \nonumber\ee

Note that the path $\widetilde{h}_t(\widetilde{x}_0)$ in $\widetilde{X}$ is a lift of the path $h_t(x_0)$ in $X$: $$p\circ\widetilde{h}_t(\widetilde{x}_0)=h_t\circ p(\widetilde{x}_0)=h_t(x_0).$$
Since $$[h_t(x_0)]=\alpha\not\in H=p_{\ast}\big(\pi_1(\widetilde{X},\widetilde{x}_0)\big),$$ the lifting path  $\widetilde{h}_t(\widetilde{x}_0)$ is \emph{not} a loop (\cite[Prop. 1.31]{Ha}) and hence \be\label{equality}\widetilde{x}_1:=\widetilde{h}_1(\widetilde{x}_0)\neq \widetilde{h}_0(\widetilde{x}_0)=\widetilde{x}_0.\ee
The subgroup $H$ is normal in $\pi_1(X,x_0)$ and thus the action of deck transformation group on each fiber of $p$ is transitive (\cite[Prop. 1.39]{Ha}), which means that there exists a deck transformation $\tau$ on $\widetilde{X}$ such that $\tau(\widetilde{x}_0)=\widetilde{x}_1$.
On the other hand $\widetilde{h}_1$ is also a lift of $p$ as $p\circ \widetilde{h}_1=h_1\circ p=p$. Then the unique lifting property (\cite[P. 1.34]{Ha}) implies that $\widetilde{h}_1=\tau$ as both of them send $\widetilde{x}_0$ to $\widetilde{x}_1$.

Now this $\tau=\widetilde{h}_1$ is the desired deck transformation. Firstly $\tau\neq\text{id}_{\widetilde{X}}$ by (\ref{equality}) and hence $\tau$ is fixed-point free, still due to the unique lifting property. Secondly, $\tau$ is of finite order as the deck transformation group is isomorphic to the quotient group $\pi_1(X,x_0)/H$ (\cite[Prop. 1.39]{Ha}), which is finite as $H$ is of finite index.

If $X$ is moreover a smooth (resp. complex) manifold, then the covering map $p$ is smooth (resp. holomorphic) and so is $\tau$ as locally $p$ is a diffeomorphism (resp. biholomorphism).
\end{proof}

\section{Proof of Theorem \ref{vanish1}}\label{proof of vanish1}
Denote by $b_i^{(2)}(M)$ the $L^2$-Betti numbers of a smooth manifold $M$ (see Section \ref{sectionL2Bettinumber} for more details on their definition and basic properties). The following well-known conjecture is usually attributed to Singer (\cite[\S11]{Lu}, \cite{Do}).
\begin{conjecture}[Singer Conjecture]\label{Singer conjecture}
Let $M$ be a $2n$-dimensional aspherical smooth orientable manifold. Then $b_i^{(2)}(M)=0$ whenever $i\neq n$. If moreover $M$ admits a negatively curved Riemannian metric, then $b_n^{(2)}(M)>0.$
\end{conjecture}
\begin{remark}
Conjecture \ref{Singer conjecture} particularly implies that $(-1)^n\chi(M)\geq0$ for aspherical manifolds and $(-1)^n\chi(M)>0$ when admitting a negatively curved Riemannian metric \big(see formula (\ref{L2Bettinumber})\big), which is the assertion of another well-known \emph{Hopf Conjecture}.
\end{remark}
In order to attack the Singer Conjecture for K\"{a}hler manifolds, Gromov introduced the notion of K\"{a}hler hyperbolicity (\cite{Gr}) and it was further extended to K\"{a}hler non-ellipticity by Cao-Xavier and Jost-Zuo independently (\cite{CX}, \cite{JZ}). For a very recent development around this notion we refer to \cite{BDET}. Recall that a K\"{a}hler manifold $(M,\omega)$, where $\omega$ is the K\"{a}hler form, is called \emph{K\"{a}hler hyperbolic} (resp. \emph{K\"{a}hler non-elliptic}) if $\pi^{\ast}(\omega)=d\beta$, where $\widetilde{M}\stackrel{\pi}{\longrightarrow}M$ is the universal cover, such that $\beta$ is a \emph{bounded} (resp. \emph{sub-linear}) one-form on $\big(\widetilde{M},\pi^{\ast}(\omega)\big)$. The following example illustrates the ampleness of these two notions.
\begin{example}\label{example}

\begin{enumerate}

\item
Typical examples of K\"{a}hler hyperbolic manifolds include (\cite[p. 265]{Gr})
K\"{a}her manifolds homotopy equivalent to negatively curved Riemannian manifolds and compact quotients of the bounded homogeneous symmetric domains in $\mathbb{C}^n$ are K\"{a}hler hyperbolic. Submanifolds and product manifolds of  K\"{a}hler hyperbolic manifolds are still  K\"{a}hler hyperbolic.

\item
Typical examples of K\"{a}hler non-elliptic manifolds include (\cite[p. 4]{JZ}) K\"{a}hler hyperbolic manifolds, non-positively curved K\"{a}hler manifolds and holomorphic immersed submanifolds of complex tori.
\end{enumerate}
\end{example}

The following vanishing-type results due to Gromov, Cao-Xavier and Jost-Zuo (\cite{Gr}, \cite{CX}, \cite{JZ}) provide an affirmative solution to the Singer Conjecture and hence the Hopf conjecture for non-positively curved and negative curved K\"{a}hler manifolds.

\begin{theorem}[Gromov, Cao-Xavier, Jost-Zuo]\label{Gromov}
Let $M$ be an $n$-dimensional K\"{a}hler non-elliptic manifold. Then the $L^2$-Betti number $b_{i}^{(2)}(M)=0$ whenever $i\neq n$. If moreover $M$ is K\"{a}hler hyperbolic, then $b_{n}^{(2)}(M)>0$.
\end{theorem}
\begin{remark}
By combining Gromov's idea with some special properties of the $\chi_y$-genus, the author deduced in \cite[Thm 2.1]{Li1} that K\"{a}hler hyperbolic manifolds indeed satisfy a family of optimal Chern number inequalities and the first one is exactly $(-1)^nc_n\geq n+1$ , which is an improved inequality expected by the Hopf conjecture.
\end{remark}

With the facts above and in Section \ref{sectionL2Bettinumber} in hand, we can now show the following fact.
\begin{proposition}\label{vanish1-prop}
Let $M$ be an $n$-dimensional K\"{a}hler manifold satisfying $b_{i}^{(2)}(M)=0$ whenever $i\neq n$. Then the Euler characteristic of $M$ is zero if and only if $\chi_y(M)=0$.
\end{proposition}
\begin{proof}
One direction is obvious as $\chi_y(M)\big|_{y=-1}$ is the Euler characteristic. For the converse direction, we first note that the hypotheses $b_{i}^{(2)}(M)=0$ whenever $i\neq n$, together with the $L^2$-Hodge decomposition (\ref{hodge decomposition}), imply that
\be\label{1}h^{p,q}_{(2)}(M)=0,\qquad\text{whenever $p+q\neq n$}.\ee

On the other hand,
\be\begin{split}
\chi(M)&=\sum_{i=0}^{2n}(-1)^ib_i^{(2)}(M)\qquad\big(\text{by (\ref{L2Bettinumber})}\big)\\
&=(-1)^nb_n^{(2)}(M)\qquad\big(\text{by hypothesis condition}\big)\\
&=(-1)^{n}\sum_{p=0}^{n}h^{p,n-p}_{(2)}(M)\qquad\big(\text{by $L^2$-Hodge decomposition (\ref{hodge decomposition}})\big).
\end{split}
\ee
Thus under the hypothesis, the Euler characteristic $\chi(M)=0$ implies that
\be\label{2}h^{p,n-p}_{(2)}(M)=0,\qquad\text{for all $0\leq p\leq n$}.\ee
Combining (\ref{1}) with (\ref{2}) yields that all the $L^2$-Hodge numbers vanish, and hence so are $\chi^p(M)$ due to Lemma \ref{L2expression}.
\end{proof}
Now we are ready to prove the following consequence, which is exactly Theorem \ref{vanish1}.
\begin{corollary}[=Theorem \ref{vanish1}]\label{vanish1-coro}
Let $M$ be a K\"{a}hler non-elliptic manifold. Then the Euler characteristic of $M$ vanishes if and only if $\chi_y(M)=0$. In particular, a non-positively curved K\"{a}hler manifold $M$ with non-trivial $\text{Z}\big(\pi_1(M)\big)$ has vanishing $\chi_y$-genus: $\chi_y(M)=0$.
\end{corollary}
\begin{proof}
The first assertion follows from Theorem \ref{Gromov} and Proposition \ref{vanish1-prop}. For the second assertion, only note that a non-positively curved K\"{a}hler manifold $M$ is both non-elliptic (Example \ref{example}) and aspherical. Asphericity and Theorem \ref{JG} imply that $\chi(M)=0$ and hence $\chi_y(M)=0$.
\end{proof}

\section{Proof of Theorem \ref{vanish2}}\label{proof of vanish2}Various notions and basic facts in algebraic geometry used in this section can be found in \cite{La1} and \cite{La2}.

%Let $E\longrightarrow M$ be a rank $r$ holomorphic vector bundle over a K\"{a}hler manifold with Chern classes $c_i(E)$. Let $\lambda=(\lambda_1,\lambda_2,\ldots,\lambda_r)$ be a partition: $$\lambda_1\geq\lambda_2\geq\cdots\geq\lambda_r\geq0,\qquad\lambda_i\in\mathbb{Z}_{\geq0},$$ whose weight $|\lambda|:=\sum_{j=1}^{r}\lambda_j$. The \emph{Schur polynomial} $S_{\lambda}=S_{\lambda}\big(c_1(E),\ldots,c_r(E)\big)$ with respect to the partition $\lambda$ is defined by
%$$S_{\lambda}\big(c_1(E),\ldots,c_r(E)\big):=
%\det\Big(c_{\lambda_i-i+j}(E)\Big)_{1\leq i,j\leq r}\in H^{|\lambda|,|\lambda|}(M).$$

In \cite{DPS} Demailly, Peternell and Schneider systematically investigated numerically-effective (\emph{nef} for short) vector bundles over K\"{a}hler manifolds and paid special attention to those K\"{a}hler manifolds whose (holomorphic) tangent bundles are nef. Among other things, they obtained that
%all the Chern numbers of a nef vector bundles on an $n$-dimensional K\"{a}hler manifolds are nonnegative and bounded from above by the Chern number $c_1^n$ (\cite[Coro. 2.6]{DPS}). Inspired by this upper bound and dual to this, the author and Zheng showed the following (\cite[Thm 2.9]{LZ})
%\begin{theorem}[Li-Zheng]
%Let $E\longrightarrow M$ be a rank $r$ nef vector bundle over an $n$-dimensional K\"{a}hler manifold with Chern classes $c_i(E)$. Then
%\be\label{FL inequality}\int_MS_{\lambda}\big(c_1(E),\ldots,c_r(E)\big)\geq0,
%\qquad\text{for all partitions $\lambda$ of weight $n$}.\ee
%\end{theorem}
%\begin{remark}
%The strict inequalities of (\ref{FL inequality}) were first showed by Fulton and Lazarsfeld in \cite{FL} to characterize positive polynomials of Chern classes on ample vector bundles. The author showed in \cite[Prop. 3.1]{Li2} the same (\ref{FL inequality}), with Chern forms in place of Chern classes, on \emph{Bott-Chern positive} vector bundles and gave some applications.
%\end{remark}
(\cite[Coro. 5.5]{DPS}) a K\"{a}hler manifold $M$ with nef tangent bundle $T_M$ is Fano, i.e., the anti-canonical line bundle $K_M^{-1}$ is ample, if and only if its Todd genus $\chi^{0}(M)>0$. For related applications of this result, we refer to \cite[Thm 1.2]{Yang} and \cite[Thm 3.2]{LZ} and the comments therein.

Inspired by this result, Qi Zhang conjectured a dual version (\cite[p. 779]{Zha}): the canonical line bundle $K_M$ of a K\"{a}hler manifold $M$ with nef cotangent bundle $T_M^{\ast}$ is ample if and only if the signed Todd genus $(-1)^n\chi^{0}(M)>0$. He showed the ``if" part for all dimensions $n$ and the ``only if" part when $n\leq4$ (\cite[Thm 4]{Zha}). For our purpose, we state a slight variant of Zhang's result as follows and indicate a proof for the reader's convenience.
\begin{theorem}[Qi Zhang]\label{vanish2-thm}
Let $M$ be a K\"{a}hler manifold with nef $T_M^{\ast}$ and dimension $n\leq4$. If the Chern number $(-c_1)^n>0$, then the signed Todd genus $(-1)^n\chi^0(M)>0$.
\end{theorem}
\begin{proof}
The ideas of the proof are essentially scattered in \cite{Zha}, although not stated as above. Below some points of our proof are different from those in \cite{Zha}.  

\emph{Claim 1: $M$ is of general type and projective.}
Indeed, the nefness of $T_M^{\ast}$ implies that of $K_M$ (\cite[p. 24]{La2}), which, together with the condition $(-c_1)^n>0$, implies that $K_M$ is big (\cite[Thm 0.5]{DP}). Then $M$ is Moishezon and hence projective as it is K\"{a}hler (\cite[p. 95]{MM}).

\emph{Claim 2: $K_M$ is ample, i.e., the first Chern class $c_1(M)<0$.} The nefness of $T_M^{\ast}$ and the projectivity of $M$ implies that it contains no rational curves. Hence $K_M$ is ample (\cite[p. 785]{Zha}, \cite[p. 219]{De}, \cite[p.1481-1482]{CY}).

For simplicity we write $c_i:=c_i(M)$. Note that the ampleness of $K_M$ indeed holds for all \emph{n}. Then the Miyaoka-Yau Chern number inequality (\cite{Yau}) reads
\be\label{Chern number inequality}
c_2(-c_1)^{n-2}\geq \frac{n}{2(n+1)}(-c_1)^n>0.
\ee
When $n=2$ or $3$, $(-1)^n\chi^0(M)=\frac{1}{12}(c_1^2+c_2)$ or $-\frac{1}{24}c_1c_2$ respectively (\cite[p. 14]{Hi}). Its positivity follows from (\ref{Chern number inequality}) directly.

When $n=4$, (\ref{Chern number inequality}) reads
\be\label{3}5c_2c_1^2\geq2c_1^4.\ee
On the other hand, we have (cf. \cite[Thm 2.9]{LZ})
\be\label{4}c_1c_3-c_4\geq0,\qquad c_2^2\geq0.\ee
Therefore,
\be\begin{split}
(-1)^4\chi^0(M)&=\frac{1}{720}(-c_4+c_1c_3+3c_2^2+
4c_1^2c_2-c_1^4)\qquad\big(\text{\cite[p. 14]{Hi}}\big)\\
&=\frac{1}{720}\big[(c_1c_3-c_4)+3c_2^2+
\frac{1}{2}(5c_1^2c_2-2c_1^4)+\frac{3}{2}c_1^2c_2\big]\\
&\geq\frac{1}{480}c_1^2c_2\qquad\big(\text{by (\ref{3}) and (\ref{4})}\big)\\
&\geq\frac{1}{1200}c_1^4\qquad\big(\text{by (\ref{3})}\big)\\
&>0.\qquad\big(\text{by the hypothesis condition}\big)
\end{split}\nonumber\ee
\end{proof}
We now arrive at the proof of Theorem \ref{vanish2}.
\begin{theorem}[=Theorem \ref{vanish2}]\label{vanish2-coro}
Let $M$ be an $n$-dimensional non-positively curved K\"{a}hler manifold with non-trivial $\text{Z}\big(\pi_1(M)\big)$. Then all the Chern numbers of $M$ vanish whenever $n\leq4$.
\end{theorem}
\begin{proof}
By \cite[Coro. 3.5]{LZ}, in this situation all the signed Chern numbers are non-negative and bounded above by $(-c_1)^n$. Thus it suffices to show $(-c_1)^n=0$. Suppose on the contrary that $(-c_1)^n>0$. We have mentioned in Section \ref{introduction} that the sign of holomorphic bisectional curvature is dominated by that of Riemannian sectional curvature, while the non-negativity of the former famously implies the nefness of the cotangent bundle. Then Theorem \ref{vanish2-thm} implies that the signed Todd genus $(-1)^n\chi^0(M)>0$ and hence the $\chi_y$-genus $\chi_y(M)\neq0$. Therefore Proposition \ref{vanish1-prop} tells us that the Euler characteristic of $M$ is nonzero, which by Theorem \ref{JG} contradicts to the hypothesis of $\text{Z}\big(\pi_1(M)\big)$ being non-trivial.
\end{proof}
Now we explain that why the proof of Theorem \ref{vanish2-coro} leads to Corollary \ref{vanish3}.
As remarked in \cite[Remark 4.2]{LZ}, Conjecture \ref{question} is \emph{equivalent to} the simultaneous positivity and vanishing of all signed Chern numbers for K\"{a}hler manifold
with non-positive holomorphic bisectional curvature. Thus we have 
\begin{corollary}[=Corollary \ref{vanish3}]
Let $M$ be an $n$-dimensional non-positively curved K\"{a}hler manifold with $n\leq 4$. Then all the signed Chern numbers of $M$ are either positive or zero. Thus Conjecture \ref{question} holds true for these manifolds.
\end{corollary}
\begin{proof}
By \cite[Coro. 3.5]{LZ}, it suffices to show that when the Chern number $(-c_1)^n>0$, then the signed Euler characteristic  $(-1)^n\chi(M)>0$. This has exactly been done in the proof of Theorem \ref{vanish2-coro}.
\end{proof}

\section{Appendix}\label{appendix}
\subsection{Proof of Lemma \ref{L2expression}}\label{moredetailsonL2Hodgenumbers}
We first explain that the elliptic complexes (\ref{DC}) can be made into more compact two-step elliptic operators as follows. The choice of a Hermitian metric
on the complex manifold $(M,J)$
enables us to define the Hodge star operator $\ast$
and the formal adjoint operator
$\bar{\partial}^{\ast}=-\ast\bar{\partial}~\ast$ of the
$\bar{\partial}$-operator. Then for each $0\leq p\leq n$, we have
the following first-order elliptic operator $D_p$:
\be\label{GDC}D_p:=\bar{\partial}+\bar{\partial}^{\ast}:~\bigoplus_{\textrm{$q$
even}}\Omega^{p,q}(M)\longrightarrow\bigoplus_{\textrm{$q$
odd}}\Omega^{p,q}(M),\ee whose index by definition is
\be\begin{split}
\text{ind}(D_p)=&\text{dim}_{\mathbb{C}}(\text{ker}D_p)-
\text{dim}_{\mathbb{C}}(\text{coker}D_p)\\
=&\text{dim}_{\mathbb{C}}\bigoplus_{\text{$q$ even}}\mathcal{H}^{p,q}_{\bar{\partial}}(M)-
\text{dim}_{\mathbb{C}}\bigoplus_{\text{$q$ odd}}\mathcal{H}^{p,q}_{\bar{\partial}}(M)\\
=&\sum_{\text{$q$ even}}h^{p,q}(M)-\sum_{\text{$q$ odd}}h^{p,q}(M)\\
=&\chi^p(M),
\end{split}\nonumber\ee
where $\mathcal{H}^{p,q}_{\bar{\partial}}(M)$ are the spaces of complex-valued $\bar{\partial}$-harmonic forms and their dimensions are famously equal to $h^{p,q}(M)$.

The elliptic operators $D_p$ in (\ref{GDC}) on the Hermitian manifold $(M,g,J)$ can be lifted to its universal cover $(\widetilde{M},\widetilde{g},\widetilde{J})$: $$\widetilde{D_p}:~\bigoplus_{\text{$q$ even}}L^2\Omega^{p,q}(\widetilde{M})\longrightarrow\bigoplus_{\text{$q$ odd}}L^2\Omega^{p,q}(\widetilde{M}),$$ and one can define the \emph{$L^2$-index} of the lifted operators $\widetilde{D_p}$ by
\be
\begin{split}
\text{ind}_{\pi_1(M)}(\widetilde{D_p}):=&
\text{dim}_{\pi_1(M)}(\text{ker}\widetilde{D_p})-
\text{dim}_{\pi_1(M)}(\text{coker}\widetilde{D_p})\\
=&\text{dim}_{\pi_1(M)}\big[\bigoplus_{\text{$q$ even}}\mathcal{H}^{p,q}_{(2)}(\widetilde{M})\big]
-\text{dim}_{\pi_1(M)}\big[\bigoplus_{\text{$q$ odd}}\mathcal{H}^{p,q}_{(2)}(\widetilde{M})\big]\\
=&\sum_{\text{$q$ even}}\text{dim}_{\pi_1(M)}\mathcal{H}^{p,q}_{(2)}(\widetilde{M})
-\sum_{\text{$q$ odd}}\text{dim}_{\pi_1(M)}\mathcal{H}^{p,q}_{(2)}(\widetilde{M})
\qquad\big(\text{dim}_{\pi_1(M)}(\cdot)~\text{is additive}\big)
\\
=&\sum_{q=0}^n(-1)^qh^{p,q}_{(2)}(M).
\end{split}\nonumber\ee

The $L^2$-index theorem of Atiyah (\cite{At}) asserts that
$$\text{ind}(D_p)=\text{ind}_{\pi_1(M)}(\widetilde{D_p})
$$
and so we have the following crucial identities among $\chi^p(M)$ and the $L^2$-Hodge numbers $h^{p,q}_{(2)}(M)$ stated in (\ref{factchiphodgenumbers}):
\be
\chi^p(M)=\sum_{q=0}^n(-1)^qh^{p,q}_{(2)}(M).\nonumber\ee
\begin{remark}
Note that the d-bar operator $\bar{\partial}$ can also be defined for \emph{almost-complex} manifolds $(M,J)$, and the almost-complex structure $J$ is integrable if and only if $\bar{\partial}^2\equiv0$, i.e., (\ref{DC}) are \emph{complexes}.  So another advantage of using (\ref{GDC}) rather than (\ref{DC}) is that (\ref{GDC}) and hence the $\chi_y$-genus can be defined for general \emph{almost-complex} manifolds.
\end{remark}

\subsection{$L^2$-Betti numbers and $L^2$-Hodge decomposition}\label{sectionL2Bettinumber}
As in the discussions in Section \ref{read}, we can similarly define the $L^2$-Betti numbers $b_i^{(2)}(M)$ of a general $k$-dimensional compact smooth orientable manifold $M$ to be the Von Neumann dimension of the space of $L^2$ harmonic $i$-forms $\mathcal{H}^{i}_{(2)}(\widetilde{M})$ with respect to $\pi_1(M)$ under an arbitrarily chosen Riemannian metric:
$$b_i^{(2)}(M):=\text{dim}_{\pi_1(M)}
\mathcal{H}^{i}_{(2)}(\widetilde{M})\in\mathbb{R}_{\geq0},\qquad0\leq i\leq k.$$

The Euler characteristic $\chi(M)$ is the index of the following elliptic operator
\be\label{GDC2}d+d^{\ast}:~\bigoplus_{\textrm{$i$
even}}\Omega^{i}(M)\longrightarrow\bigoplus_{\textrm{$i$
odd}}\Omega^{i}(M).\ee
Parallel to the discussions in the above subsection, we can lift (\ref{GDC2}) to the universal cover $\widetilde{M}$ and apply the $L^2$-index theorem of Atiyah (\cite{At}) to yield
\be\label{L2Bettinumber}\chi(M)=\sum_{i=0}^k(-1)^ib_i^{(2)}(M).\ee

When $M$ is K\"{a}hler with complex dimension $n$, the $L^2$ Hodge-decomposition (cf. \cite[\S1.2]{Gr}) implies that
\be\label{hodge decomposition}b_i^{(2)}(M)=\sum_{p+q=i}h^{p,q}_{(2)}(M).\ee

\end{document}